\documentclass[sn-mathphys,Numbered]{sn-jnl}

\usepackage{graphicx}%
\usepackage{multirow}%
\usepackage{amsmath,amssymb,amsfonts}%
\usepackage{amsthm}%
\usepackage{mathrsfs}%
\usepackage[title]{appendix}%
\usepackage{xcolor}%
\usepackage{textcomp}%
\usepackage{manyfoot}%
\usepackage{booktabs}%
\usepackage{algorithm}%
\usepackage{algorithmicx}%
\usepackage{algpseudocode}%
\usepackage{listings}%

 \newtheorem{thm}{Theorem}[section]
 \newtheorem{cor}[thm]{Corollary}
 
 \newtheorem{prop}[thm]{Proposition}
 \theoremstyle{definition}
 \newtheorem{defn}[thm]{Definition}
 \theoremstyle{remark}
 \newtheorem{rem}[thm]{Remark}
 \newtheorem{ex}[thm]{Example}
 \numberwithin{equation}{section}

\newcommand{\NN}{\mathbb{N}}

\begin{document}

\title[A three point extension of Chatterjea's fixed point theorem]{A three point extension of Chatterjea's fixed point theorem with at most two fixed points}


\author[1]{\fnm{Ravindra K.} \sur{Bisht}}\email{ravindra.bisht@yahoo.com}

\author*[2]{\fnm{Evgeniy} \sur{Petrov}}\email{eugeniy.petrov@gmail.com}

\affil[1]{
\orgdiv{Department of Mathematics},
\orgname{National Defence Academy},
\orgaddress{
\city{Pune},
\state{Khadakwasla},
\country{India}}}

\affil*[2]{
\orgdiv{Function Theory Department},
\orgname{Institute of Applied Mathematics and Mechanics of the NAS of Ukraine}, \orgaddress{\street{Batiuka str. 19},
\city{Slovyansk},
\postcode{84116},
\country{Ukraine}}}



\abstract{In this paper, we introduce a new category of mappings within metric spaces, specifically focusing on three-point analogs of the well-established Chatterjea type mappings. We demonstrate that Chatterjea type mappings and their three-point analogs are different classes of mappings. A fixed point theorem for generalized Chatterjea type mappings is established. It is shown that these mappings are continuous at fixed point. Connections between generalized Chatterjea type mappings, generalized Kannan type mappings, Chatterjea type mappings, and mappings contracting perimeters of triangles are found. Additionally, we derive two fixed point theorems for generalized Chatterjea type mappings in metric spaces, even when completeness is not mandatory.}

\keywords{fixed point theorem, Chatterjea type mapping, generalized Chatterjea type mapping, metric space}


\pacs[MSC Classification]{Primary 47H10; Secondary 47H09}

\maketitle

\section{Introduction}
Over the span of more than a century, numerous efforts have been made to extend and generalize Banach's contraction principle. This principle, being a fundamental concept, finds applicability in a wide range of problems within the realms of science and engineering.

In 1968, Kannan \cite{Ka68} proved an interesting fixed point theorem that played a crucial role in advancing the theory of fixed points for generalized contractive mappings. This theorem was quickly followed by numerous papers on contractive mappings. The Kannan fixed point theorem also gave rise to the famous question regarding the continuity of contractive mappings at their fixed points.

In ~\cite{Ch72}, Chatterjea presented a result in 1972 that offers a fixed point for mappings even in cases where they are discontinuous within their domain of definition: let $T\colon X\to X$ be a mapping on a complete metric space $(X,d)$ with
  \begin{equation}\label{e0}
   d(Tx,Ty)\leqslant \gamma (d(x,Ty)+d(y,Tx))
  \end{equation}
where $0\leqslant \gamma<\frac{1}{2}$ and $x,y \in X$. Then $T$ has a unique fixed point.

Similar to Kannan's fixed-point theorem, condition (\ref{e0}) also ensures continuity for the mapping $T$ at the fixed point \cite{Rh88}.  Additionally, the Banach Contraction Principle, Kannan mappings, and Chatterjea mappings are each independent of one another, except for sharing the uniqueness of the fixed point.

In the realm of Fixed Point Theory, distinctions can be made among various categories of generalizations of the Chatterjea theorem. In the first scenario, the contractive nature of the mapping is relaxed, as demonstrated, for instance, in~\cite{C13,Deb21,Kad16,Be07,JKR11,Su18}. The second scenario involves  the relaxation of the topology, as explored, for example, in~\cite{AJS18}. The third case encompasses theorems tailored for Chatterjea-type multivalued mappings, as discussed in works such as~\cite{C19,T22}. Lastly, in the fourth case, nontrivial extensions are considered within a weaker or more extended structure of the metric space, offering a comprehensive exploration of the diverse avenues of generalization in this field \cite{Har11,KI19,Ko17,Ma16,Ber21}.

In 2023, Petrov introduced a new class of mappings, characterized as a three-point analog of the Banach contraction principle, and referred to it as mappings that contract the perimeters of triangles~\cite{P23}.

 \begin{defn}\label{d0}
Let $(X,d)$ be a metric space with $|X|\geqslant 3$. We say that $T\colon X\to X$ is a \emph{mapping contracting perimeters of triangles} on $X$ if there exists $\alpha\in [0,1)$ such that the inequality
  \begin{equation}\label{mcpt}
   d(Tx,Ty)+d(Ty,Tz)+d(Tx,Tz) \leqslant \alpha (d(x,y)+d(y,z)+d(x,z))
  \end{equation}
  holds for all three pairwise distinct points $x,y,z \in X$.
\end{defn}

\begin{rem}
It is crucial to highlight the significance of the condition requiring pairwise distinct elements for $x, y, z \in X$. Without this stipulation, the definition becomes indistinguishable from the Banach contraction principle.
\end{rem}

In metric fixed-point theory, geometric elegance emerges by substituting terms with various combinations of distinct distances, which play a pivotal role. In the context of the six points $x, y, z, Tx, Ty, Tz$ outlined in Definition \ref{d0}, there are $\binom{6}{2}$ possible combinations, resulting in 15 distances when considered two at a time. One such combination of three distances, specifically $d(Tx, Ty) + d(Ty, Tz) + d(Tz, Tx)$, is presented on the left-hand side (LHS) in (\ref{mcpt}). On the other hand, the right-hand side (RHS) expression consists of a combination of three distances, i.e., $d(x, y) + d(y, z) + d(x, z)$. In general, the remaining 12 distances (except $d(Tx, Ty) + d(Ty, Tz) + d(Tz, Tx)$ fixed on LHS) undergo rotation in a manner that ensures the preservation of the symmetric condition of the metric space and provides motivation for various classes of three point analogous mappings, similar to the well-established two-point versions.

Following~\cite{P23}, the authors in~\cite{EB23} presented a three-point analogue of the Kannan type mappings.

\begin{defn}\label{d1}
Let $(X,d)$ be a metric space with $|X|\geqslant 3$. We say that $T\colon X\to X$ is a \emph{generalized Kannan type mapping} on $X$ if there exists $\lambda\in [0,\frac{2}{3})$ such that the inequality
  \begin{equation}\label{gktm}
   d(Tx,Ty)+d(Ty,Tz)+d(Tx,Tz) \leqslant \lambda (d(x,Tx)+d(y,Ty)+d(z,Tz))
  \end{equation}
  holds for all three pairwise distinct points $x,y,z \in X$.
\end{defn}

The generalized Kannan type condition, as expressed in equation (\ref{gktm}), involves substituting the terms $d(x,y) + d(y,z) + d(z,x)$, found in the RHS of the mapping of contracting perimeters of triangles (\ref{mcpt}), with $d(x, Tx) + d(y, Ty) + d(z, Tz)$. It is important to mention that the classes of Kannan type mappings and generalized Kannan type mappings are independent \cite{EB23}.\\

Motivated by the insights from~\cite{P23} and utilizing the remaining six distances among the points $x, y, z, Tx, Ty, Tz$, denoted as $d(x, Ty) + d(x, Tz) + d(y, Tx) + d(y, Tz) + d(z, Tx) + d(z, Ty)$ on the RHS of the generalized Kannan type condition (\ref{gktm}), we introduce a generalized Chatterjea type mapping as follows:

\begin{defn}\label{d2}
Let $(X,d)$ be a metric space with $|X|\geqslant 3$. We shall say that $T\colon X\to X$ is a \emph{generalized Chatterjea type mapping} on $X$ if there exists $\gamma\in [0,\frac{1}{3})$ such that the inequality
  \begin{equation}\label{gctm}
\begin{aligned}
  &d(Tx,Ty) + d(Ty,Tz) + d(Tx,Tz) \\
  &\leqslant \gamma (d(x,Ty) + d(x,Tz) + d(y,Tx) + d(y,Tz) + d(z,Tx) + d(z,Ty))
\end{aligned}
\end{equation}
  holds for all three pairwise distinct points $x,y,z \in X$.
\end{defn}

The fundamental distinction in the definition of generalized Chatterjea type mappings lies in the mapping of three points in space rather than two. Additionally, a requisite condition is imposed to preclude these mappings from possessing periodic points of prime period 2.

In Section~\ref{sec2}, {we explore the connections between generalized Chatterjea type mappings, generalized Kannan type mappings, Chatterjea type mappings,} and mappings contracting perimeters of triangles. Moving to Section~\ref{sec3}, we establish the main result, Theorem~\ref{t1}, which constitutes a fixed-point theorem for generalized Chatterjea type mappings, asserting that the number of fixed points is at most two, and we additionally demonstrate the continuity of these mappings at the fixed points. Finally, in Section~\ref{sec5}, inspired by Kannan's work~\cite{Ka69}, we derive two fixed-point theorems for generalized Chatterjea type mappings, relaxing the completeness requirement for the underlying space.

\section{Some properties of generalized Chatterjea type mappings}\label{sec2}

In this section, we explore the relationship between generalized Chatterjea type mappings, Chatterjea type mappings, mappings contracting perimeters of triangles, and generalized Kannan type mappings.

\begin{prop}\label{p2.1}
Chatterjea  type mappings with $\gamma\in[0,\frac{1}{3})$ are generalized Chatterjea type mappings.
\end{prop}
\begin{proof}
Let $(X,d)$ be a metric space with $|X|\geqslant 3$, $T\colon X\to X$ be a Chatterjea type mapping and let $x,y,z\in X$ be pairwise distinct. Consider inequality~(\ref{e0}) for the pairs $y, z$ and  $z, x$:
\begin{equation}\label{e01}
   d(Ty,Tz)\leqslant \gamma (d(y,Tz)+d(z,Ty)),
\end{equation}
\begin{equation}\label{e02}
   d(Tz,Tx)\leqslant \gamma (d(z,Tx)+d(x,Tz)).
\end{equation}

Summarizing the left and the right parts of inequalities~(\ref{e0}), ~(\ref{e01}) and~(\ref{e02}) we obtain
\begin{align*}
   d(Tx,Ty)+&d(Ty,Tz)+d(Tz,Tx)\\
&\leqslant \gamma (d(x,Ty)+d(x,Tz)+d(y,Tx)+d(y,Tz)+d(z,Tx)+d(z,Ty)).
\end{align*}
Hence, we get the desired assertion.
\end{proof}

Recall that for a given metric space $X$, a point $x \in X$ is said to be an \emph{accumulation point} of  $X$ if every open ball centered at $x$ contains infinitely many points of $X$.

\begin{prop}\label{p1}
Let $(X,d)$ be a metric space and let  $T\colon X\to X$ be a generalized Chatterjea type metric with some  $\gamma\in [0,\frac{1}{4})$. If $x$ is an accumulation point of  $X$ and $T$ is continuous at $x$, then the inequality
\begin{equation}\label{w1}
 d(Tx,Ty)\leqslant \frac{\gamma}{1-\gamma}\left(2d(x,Ty)+d(y,Tx)\right)
\end{equation}
holds for all points $y\in X$.
\end{prop}

\begin{proof}
Let $x\in X$ be an accumulation point and let $y\in X$. If $y=x$, then clearly~(\ref{w1}) holds. Let now $y\neq x$. Since $x$ is an accumulation point, then there exists a sequence $z_n\to x$ such that $z_n\neq x$, $z_n\neq y$ and all $z_n$ are different.
Hence, by~(\ref{gctm}) the inequality
\begin{align*}
   d(Tx,Ty)&+d(Ty,Tz_n)+d(Tx,Tz_n)\\
   &\leqslant \gamma (d(x,Ty)+ d(x,Tz_n)+d(y,Tx)+d(y,Tz_n)+d(z_n,Tx)+d(z_n,Ty))
\end{align*}
holds for all $n\in \NN$. Since $z_n\to x$ and  $T$ is continuous at $x$ we have $Tz_n \to Tx$. Since every metric is continuous we have $d(z_n, Tz_n) \to d(x,Tx)$. Hence, we get
\begin{align*}
   d(Tx,Ty)&+d(Ty,Tx)\\
   &\leqslant \gamma (d(x,Ty)+ d(x,Tx)+d(y,Tx)+d(y,Tx)+d(x,Tx)+d(x,Ty))\\
   &\leqslant 2\gamma (d(x,Ty)+ d(y,Tx)+ d(x,Tx))\\
   &\leqslant 2\gamma (d(x,Ty)+ d(y,Tx)+ d(x,Ty)+d(Ty,Tx)),
\end{align*}

which is equivalent to  ~(\ref{w1}).
\end{proof}

The following corollary establishes a very interesting property of generalized Chatterjea type mappings.

\begin{cor}\label{cor1}
Let $(X,d)$ be a metric space, $T\colon X\to X$ be a continuous generalized Chatterjea type mapping with $\gamma\in [0,\frac{1}{4})$ and let all points of $X$ are accumulation points. Then $T$ is a Chatterjea type mapping.
\end{cor}
\begin{proof}
According to Proposition~\ref{p1}, inequality~(\ref{w1}) holds as well as the inequality
\begin{equation}\label{w2}
 d(Tx,Ty)\leqslant \frac{\gamma}{1-\gamma}\left(2d(y,Tx)+d(x,Ty)\right).
\end{equation}
Summarizing the left and the right parts of ~(\ref{w1}) and ~(\ref{w2}) and dividing both parts of the obtained inequality by $2$ we get
\begin{equation*}
 d(Tx,Ty)\leqslant \frac{3\gamma}{2(1-\gamma)}\left(d(x,Ty)+d(y,Tx)\right).
\end{equation*}
Since $\gamma\in [0,\frac{1}{4
})$, we have $\frac{3\gamma}{2(1-\gamma)} \in [0,\frac{1}{2})$, which completes the proof.
\end{proof}

\begin{prop}
Let $(X, d)$ be a metric space with $|X| \geq 3$, and $T\colon X \to X$ be a mapping contracting perimeters of triangles with a constant $0\leqslant \alpha < \frac{1}{4}$. Then $T$ is a generalized Chatterjea type mapping with respect to the metric $d$.
\end{prop}
\begin{proof}
Applying several times triangle inequality to the right part of~(\ref{mcpt}) we get
\begin{equation*}
\begin{aligned}
    &d(Tx,Ty) + d(Ty,Tz) + d(Tx,Tz) \\
    &\leqslant \alpha (d(x,Ty) + d(Ty,Tx) + d(Tx,y) + d(y,Tz) \\
    &+ d(Tz,Ty) + d(Ty,z)+ d(z,Tx) + d(Tx,Tz) + d(Tz,x)).
\end{aligned}
\end{equation*}
Rearranging this inequality, we obtain
\begin{equation*}
\begin{aligned}
    &d(Tx, Ty) + d(Ty, Tz) + d(Tz, Tx) \\
    &\leq \frac{\alpha}{1-\alpha} \left(d(x, Ty) + d(x, Tz) + d(y, Tx) + d(y, Tz) + d(z, Tx) + d(z, Ty)\right)
\end{aligned}
\end{equation*}
for  all three pairwise distinct points $x, y, z\in X$.

Since $0\leqslant \alpha <\frac{1}{2}$ we get  $0 \leqslant \gamma = \frac{\alpha}{1 - \alpha} < \frac{1}{3}$. Thus, $T$ is a generalized Chatterjea type mapping.
\end{proof}

\begin{prop}
Let $(X, d)$ be a metric space with $|X| \geq 3$, and $T\colon X \to X$ be a generalized Kannan mapping with a constant $0\leqslant \lambda < \frac{2}{5}$. Then $T$ is a generalized Chatterjea type mapping with respect to the metric $d$.
\end{prop}
\begin{proof}
Applying triangle inequality to the right part of~(\ref{gktm}), we get
\begin{equation*}
\begin{aligned}
    &d(Tx,Ty) + d(Ty,Tz) + d(Tx,Tz) \\
    &\leqslant \lambda (d(x,Ty) + d(Ty,Tx) + d(y,Tz) + d(Tz,Ty) \\
    &+ d(z,Tx) + d(Tx,Tz)).
\end{aligned}
\end{equation*}
Rearranging this inequality, we obtain
\begin{equation} \label{ff}
\begin{aligned}
    &d(Tx, Ty) + d(Ty, Tz) + d(Tz, Tx) \\
    &\leq \frac{\lambda}{1-\lambda} (d(x, Ty) + d(y, Tz) + d(z, Tx))
\end{aligned}
\end{equation}
for  all three pairwise distinct points $x, y, z\in X$.\\

Analogously, we can get the inequality
\begin{equation} \label{gg}
\begin{aligned}
    &d(Tx, Ty) + d(Ty, Tz) + d(Tz, Tx) \\
    &\leq \frac{\lambda}{1-\lambda} (d(y, Tx) + d(z, Ty) + d(x, Tz))
\end{aligned}
\end{equation}
for  all three pairwise distinct points $x, y, z\in X$.

Adding (\ref{ff}) and (\ref{gg}), we get
\begin{equation*}
\begin{aligned}
    &d(Tx, Ty) + d(Ty, Tz) + d(Tz, Tx) \\
    &\leq \frac{\lambda}{2(1-\lambda)} (d(x, Ty) + d(x,Tz) +d(y, Tx) + d(y, Tz)+ d(z, Tx)+d(z, Ty)).
\end{aligned}
\end{equation*}

Since $0\leqslant \lambda <\frac{2}{5}$ we get  $0 \leqslant \gamma = \frac{\lambda}{2(1 - \lambda)} < \frac{1}{3}$. Thus, $T$ is a generalized Chatterjea type mapping.
\end{proof}

\begin{ex}\label{exa4}
Let $(X,d)$ be a metric space such that $X=\{x,y,z\}$, $d(x,z)=d(y,z)=4$, $d(x,y)=1$ and let a mapping $T\colon X\to X$ be such that $Tx=x$, $Ty=y$ and $Tz=y$. It is clear that inequality~(\ref{e0}) does not hold for any $0\leqslant \gamma <\frac{1}{2}$:
$$
1\leqslant \gamma (1+1).
$$
Thus, $T$ is not a Chatterjea type mapping. But inequality~(\ref{gctm}) holds with, e.g., $\gamma=\frac{3}{11}<\frac{1}{3}$:
$$
1+0+1\leqslant \gamma(1+1+1+0+4+4).
$$
Thus, $T$ is a generalized Chatterjea type mapping.
\end{ex}

Example~\ref{exa4} shows that the class of generalized Chatterjea type mappings does not coincide with the class of Chatterjea type mappings. On the other hand Proposition~\ref{p2.1} states that Chatterjea  type mappings with $\gamma\in[0,\frac{1}{3})$ are generalized Chatterjea type mappings. This leads us to the following.

\textbf{Open problem.} Do there exist Chatterjea type mappings with $\gamma\in[\frac{1}{3}, \frac{1}{2})$ which are not generalized Chatterjea type mappings?

\section{Main results}\label{sec3}

Let $T$ be a mapping on the metric space $X$. A point $x\in X$ is called a \emph{periodic point of period $n$} if $T^n(x) = x$. The least positive integer $n$ for which $T^n(x) = x$ is called the prime period of $x$, see, e.g.,~\cite[p.~18]{De22}. In particular, the point $x$ is of prime period $2$ if $T^2(x)=x$ and $Tx\neq x$.

The following theorem is the main result of this paper:
\begin{thm}\label{t1}
Let $(X,d)$, $|X|\geqslant 3$, be a complete metric space and let the mapping $T\colon X\to X$ satisfy the following two conditions:
\begin{itemize}
  \item [(i)] $T$ does not possess periodic points of prime period $2$.
  \item [(ii)] $T$ is a generalized Chatterjea type mapping on $X$.
\end{itemize}
Then $T$ has a fixed point. The number of fixed points is at most two.
\end{thm}

\begin{proof}
Let $x_0\in X$, $Tx_0=x_1$, $Tx_1=x_2$, \ldots, $Tx_n=x_{n+1}$, \ldots. Suppose that $x_n$ is not a fixed point of the mapping $T$ for every $n=0,1,...$.
Since $x_{n-1}$ is not fixed, then $x_{n-1}\neq x_n=Tx_{n-1}$. By condition (i) $x_{n+1}=T(T(x_{n-1}))\neq x_{n-1}$ and by the supposition that $x_{n}$ is not fixed we have $x_n\neq x_{n+1}=Tx_n$. Hence, $x_{n-1}$, $x_n$ and $x_{n+1}$ are pairwise distinct. Let us set in~(\ref{gctm}) $x=x_{n-1}$, $y=x_n$, $z=x_{n+1}$.
Then
\begin{align*}
&d(Tx_{n-1},Tx_n)+d(Tx_n,T{x_{n+1}})+d(Tx_{n-1}, Tx_{n+1}) \\
&\leqslant \gamma(d(x_{n-1},Tx_{n})+d(x_{n-1},Tx_{n+1})+d(x_{n}, Tx_{n-1}) \\
&\quad+ d(x_{n},Tx_{n+1})+d(x_{n+1},Tx_{n-1})+d(x_{n+1}, Tx_{n}))
\end{align*}
and
\begin{align*}
&d(x_{n},x_{n+1})+d(x_{n+1},x_{n+2})+d(x_{n+2}, x_{n})\\
&\leqslant \gamma(d(x_{n-1},x_{n+1})+d(x_{n-1},x_{n+2})+d(x_{n},x_{n})\\
&+d(x_{n},x_{n+2})+d(x_{n+1},x_{n})+d(x_{n+1}, x_{n+1})).
\end{align*}
Hence,
\begin{align*}
d(x_{n+1},x_{n+2}) &\leqslant \gamma(d(x_{n-1},x_{n+1})+d(x_{n+2},x_{n-1})) +(\gamma-1)[d(x_{n+1},x_{n})+d(x_{n},x_{n+2})], \\
&\leqslant \gamma(d(x_{n-1},x_{n+1})+d(x_{n+2},x_{n-1}))-(1-\gamma)[d(x_{n+1},x_{n})+d(x_{n},x_{n+2})].
\end{align*}

Using the triangle inequality $d(x_{n+1}, x_{n+2})\leqslant d(x_n, x_{n+1})+d(x_{n+2},x_n)$, and in view of $-(1-\gamma)[d(x_n, x_{n+1})+d(x_{n+2},x_n)] \leqslant -(1-\gamma)d(x_{n+1}, x_{n+2})$, we get
$$
(2-\gamma)d(x_{n+1},x_{n+2})
\leqslant \gamma(d(x_{n-1},x_{n+1})+d(x_{n+2},x_{n-1})).
$$
In view of triangle inequality this further implies
\begin{align*}
(2-\gamma)d(x_{n+1},x_{n+2}) &\leqslant \gamma (d(x_{n-1},x_n)+d(x_n,x_{n+1})+d(x_{n+2},x_{n+1}) \\
&\quad+d(x_{n+1},x_{n})+ d(x_{n},x_{n-1})).
\end{align*}

Therefore
$$
(2-2\gamma)d(x_{n+1},x_{n+2})\leqslant 2\gamma(d(x_{n-1},x_n)+d(x_n,x_{n+1}))
$$
and
$$
d(x_{n+1},x_{n+2})\leqslant \frac{2\gamma}{1-\gamma} \max\{d(x_{n-1}, x_{n}),d(x_{n}, x_{n+1})\}.
$$
Let $\alpha = \frac{2\gamma}{1-\gamma}$. Using the relation $\gamma \in [0,\frac{1}{3})$, we get $\alpha \in [0,1)$. Further,
\begin{equation}\label{e2}
d(x_{n+1},x_{n+2})\leqslant \alpha \max\{d(x_{n-1}, x_{n}),d(x_{n}, x_{n+1})\}.
\end{equation}

Set $a_n=d(x_{n-1},x_n)$, $n=1,2,\ldots,$ and let
$a=\max\{a_{1},a_{2}\}$.
Hence and by~(\ref{e2}) we obtain
$$
a_1\leqslant a, \, \,
a_2\leqslant a, \, \,
a_3\leqslant \alpha a, \, \,
a_4\leqslant \alpha a, \, \,
a_5\leqslant \alpha^2 a, \, \,
a_6\leqslant \alpha^2 a, \,\,
a_7\leqslant \alpha^3 a, \,
\ldots .
$$
Since $\alpha <1$, it is clear that the inequalities
$$
a_1\leqslant a, \, \,
a_2\leqslant a, \, \,
a_3\leqslant \alpha^{\frac{1}{2}} a, \, \,
a_4\leqslant \alpha a, \, \,
a_5\leqslant \alpha^{\frac{3}{2}} a, \, \,
a_6\leqslant \alpha^2 a, \,\,
a_7\leqslant \alpha^{\frac{5}{2}} a, \,
\ldots
$$
also hold. That is,
\begin{equation}\label{e9}
a_n\leqslant \alpha^{\frac{n}{2}-1}a
\end{equation}
for $n=3,4,\ldots$.

Let $p\in \mathbb N$, $p\geqslant 2$. By the triangle inequality, for $n\geqslant 3$ we have
$$
d(x_n,\,x_{n+p})\leqslant d(x_{n},\,x_{n+1})+d(x_{n+1},\,x_{n+2})+\ldots+d(x_{n+p-1},\,x_{n+p})
$$
$$
=a_{n+1}+a_{n+2}+\cdots+a_{n+p} \leqslant
a(\alpha^{\frac{n+1}{2}-1}+\alpha^{\frac{n+2}{2}-1}+\cdots
+\alpha^{\frac{n+p}{2}-1})
$$
$$
=a\alpha^{\frac{n+1}{2}-1}(1+\alpha^{\frac{1}{2}}+\cdots
+\alpha^{\frac{p-1}{2}})=a\alpha^{\frac{n-1}{2}}\frac{1-\sqrt{\alpha^p}}{1-\sqrt{\alpha}}.
$$
Since by the supposition $0\leqslant\alpha<1$, then $0\leqslant \sqrt{\alpha^p}<1$ and $d(x_n,\,x_{n+p})\leqslant a\alpha^{\frac{n-1}{2}}\frac{1}{1-\sqrt{\alpha}}$. Hence, $d(x_n,\,x_{n+p})\to 0$ as $n\to \infty$ for every $p>0$. Thus, $\{x_n\}$ is a Cauchy sequence. By the completeness of $(X,d)$, this sequence has a limit $x^*\in X$.

Recall that any three consecutive element of the sequence $(x_n)$ are pairwise distinct. If $x^*\neq x_k$ for all $k\in \{1,2,...\}$, then inequality~(\ref{gctm}) holds for the pairwise distinct points  $x^*$, $x_{n-1}$ and $x_n$.
Suppose that there exists the smallest possible $k\in \{1,2,...\}$ such that $x^*=x_k$.  Let $m>k$ be such that $x^*=x_m$, then the sequence $(x_n)$ is cyclic starting from $k$ and can not be a Cauchy sequence. Hence, the points $x^*$, $x_{n-1}$ and $x_n$ are pairwise distinct at least when $n-1>k$.

Let us prove that $Tx^*=x^*$. If there exists $k\in \{1,2,...\}$ such that $x_k=x^*$, then suppose that $n-1>k$. By the triangle inequality and by inequality~(\ref{gctm}) we have
\begin{align*}
d(x^*,Tx^*) &\leqslant d(x^*,x_{n})+d(x_{n},Tx^*)\\
&=d(x^*,x_{n})+d(Tx_{n-1},Tx^*)\\
&\leqslant d(x^*,x_{n})+d(Tx_{n-1},Tx^*)+d(Tx_{n-1},Tx_{n})+d(Tx_{n},Tx^*)\\
&\leqslant d(x^*,x_{n})+\gamma(d(x_{n-1},Tx^*)+d(x_{n-1},Tx_{n})\\
&\quad+d(x^*,Tx_{n-1})+d(x^*,Tx_{n})+d(x_{n},Tx_{n-1})+d(x_{n},Tx^*))\\
&\leqslant d(x^*,x_{n})+\gamma(d(x_{n-1},Tx^*)+d(x_{n-1},x_{n+1})\\
&\quad+d(x^*,x_{n})+d(x^*,x_{n+1})+d(x_{n},x_{n})+d(x_{n},Tx^*)).
\end{align*}
Letting $n\to\infty$ we get the inequality
\begin{equation}\label{ew1}
d(x^*,Tx^*) \leqslant 2\gamma d(x^*,Tx^*).
\end{equation}
Since $\gamma\in[0,\frac{1}{3})$, we have $2\gamma\in[0,\frac{2}{3})$. Thus, for $\gamma \neq 0$ inequality~(\ref{ew1}) holds only if $d(x^*,Tx^*)$=0.  Hence, $x^*$ is the fixed point of $T$. Suppose now $\gamma=0$, then from~(\ref{gctm}) for pairwise distinct $x,y,z\in X$ we have
$$
d(Tx,Ty)+d(Ty,Tz)+d(Tx,Tz)=0.
$$
Suppose also that $x,y,z$ are not fixed. Then $Tx=Ty=Tz=t$, $t\neq x,y,z$.
From~(\ref{gctm}) for pairwise distinct $x,y,t\in X$ it follows also that
$$
d(Tx,Ty)+d(Ty,Tt)+d(Tx,Tt)=0.
$$
Hence, $Tx=Ty=Tt=t$ and $t$ is the fixed point of $T$.

Suppose that there exists at least three pairwise distinct fixed points $x$, $y$ and $z$.  Then $Tx=x$, $Ty=y$ and $Tz=z$ and it follows from~(\ref{gctm}) that
$$
d(x,y)+d(y,z)+d(x,z)\leqslant 2\gamma(d(x,y)+d(y,z)+d(x,z)),
$$
which is a contradiction for any $\gamma\in[0,\frac{1}{3})$.
\end{proof}

\begin{rem}
Suppose that under the supposition of the theorem the mapping $T$ has a fixed point $x^*$ which is a limit of some iteration sequence $x_0, x_1=Tx_0, x_2=Tx_1,\ldots$ such that $x_n\neq x^*$ for all $n=1,2,\ldots$. Then $x^*$ is a unique fixed point.
Indeed, suppose that $T$ has another fixed point $x^{**}\neq x^*$.
It is clear that $x_n\neq x^{**}$ for all $n=1,2,\ldots$. Hence, we have that the points $x^*$, $x^{**}$ and $x_n$ are pairwise distinct for all $n=1,2,\ldots$. Consider the ratio
\begin{align*}
R_n &= \frac{d(Tx^*,Tx^{**})+d(Tx^*,Tx_{n})+d(Tx^{**},Tx_{n})}
{d(x^*,Tx^{**})+d(x^*,Tx_n)+d(x^{**}, Tx^{*})+d(x^{**},Tx_n)+ d(x_n, Tx^{*})+d(x_n, Tx^{**})} \\
&= \frac{d(x^*,x^{**})+d(x^*,x_{n+1})+d(x^{**},x_{n+1})}
{d(x^*,x^{**})+d(x^*,x_{n+1})+d(x^{**}, x^{*})+d(x^{**},x_{n+1}
)+ d(x_n, x^{*})+d(x_n, x^{**})}.
\end{align*}

Taking into consideration that $d(x^*,x_{n+1})\to 0$, $d(x^{**},x_{n+1})\to d(x^{**},x^*)$ and $d(x_n,x^*)\to 0$, we obtain
$R_n\to \frac{1}{2}$ as $n\to \infty$, which contradicts to condition~(\ref{gctm}).
\end{rem}

\begin{defn}
Let $(X, d)$ be a metric space. A mapping $T\colon X \to X$ satisfying the condition
\begin{equation}\label{ar}
\lim_{n \to \infty} d(T^{n+1}(x), T^n(x)) = 0
\end{equation}
for all $x \in X$ is called asymptotically regular \cite{BP66}.
\end{defn}

\begin{rem}
Asymptotic regularity, valuable for expanding the range of applicable mappings across diverse parameter values, demonstrates limited utility in the case of generalized Chatterjea type mappings when contrasted with generalized Kannan type mappings \cite{EB23}, as evident from Corollary \ref{gkt1m}.
\end{rem}

The following proposition is nearly self-evident.
\begin{prop}\label{prop1}
Let $X$ be a finite nonempty metric space and let $T\colon X\to X$ be a self-mapping. It $T$ is asymptotically regular, then $T$ does not possess periodic points of prime period $n\geqslant 2$.
\end{prop}

This proposition and Theorem~\ref{t1} immediately lead to the following.

\begin{cor} \label{gkt1m}
Let $(X, d)$ be a complete metric space with $|X| \geqslant 3$ and let the mapping
$T\colon X \to X$ be asymptotically regular generalized Chatterjea type mapping. Then $T$ has a fixed point. The number of fixed points is at most two.
\end{cor}

Chatterjea type mappings are continuous at fixed points~\cite{Rh88}. The following proposition demonstrates that generalized Chatterjea type mappings also possess this property.

\begin{prop}
Generalized Chatterjea type mapping are continuous at fixed points.
\end{prop}

\begin{proof}
Let $(X,d)$ be a metric space with $|X|\geqslant 3$, $T\colon X\to X$ be a \emph{generalized Chatterjea type mapping} and $x^*$ be a fixed point of $T$. Let $(x_n)$ be a sequence such that $x_n\to x^*$, $x_{n}\neq x_{n+1}$ and $x_{n}\neq x^*$ for all $n$. Let us show that $Tx_n\to Tx^*$.
By~(\ref{gctm}), we have
\begin{align*}
\begin{split}
    d(Tx^*,Tx_n)+&d(Tx_n,Tx_{n+1})+d(Tx_{n+1},Tx^*) \\
    &\leqslant \gamma \bigl(d(x^*,Tx_n)+d(x^*,Tx_{n+1})+d(x_n,Tx^*) \\
    &\quad+d(x_{n},Tx_{n+1})+d(x_{n+1},Tx^*)+d(x_{n+1},Tx_{n})\bigr).
\end{split}
\end{align*}
Hence,
\begin{align*}
  \begin{split}
    d(Tx^*,Tx_n)+&d(Tx_{n+1},Tx^*) \\
    &\leqslant \gamma \bigl(d(x^*,Tx_n)+d(x^*,Tx_{n+1})+d(x_n,Tx^*) \\
    &\quad+d(x_{n},Tx_{n+1})+d(x_{n+1},Tx^*)+d(x_{n+1},Tx_{n})\bigr).
  \end{split}
  \end{align*}
By the triangle inequality we have
\begin{align*}
d(Tx^*,Tx_n)+&d(Tx_{n+1},Tx^*) \\
&\leqslant \gamma \bigl(d(x^*,Tx_n)+d(x^*,Tx_{n+1})+d(x_n,x^*) \\
&\quad+d(x_n,x^*)+d(x^*,Tx_{n+1})+d(x_{n+1},x^*) \\
&\quad+d(x_{n+1},x^*)+d(x^*,Tx_{n})\bigr)\\
&= \gamma \bigl(d(Tx^*,Tx_n)+d(Tx^*,Tx_{n+1})+d(x_n,x^*) \\
&\quad+d(x_n,x^*)+d(Tx^*,Tx_{n+1})+d(x_{n+1},x^*) \\
&\quad+d(x_{n+1},x^*)+d(Tx^*,Tx_{n})\bigr).
\end{align*}
Further,
 $$
   d(Tx^*,Tx_n)+d(Tx_{n+1},Tx^*) \leqslant \frac{2\gamma}{1-2\gamma}(d(x_n,x^*)+d(x_{n+1},x^*)).
$$
Since $d(x_n,x^*)\to 0$ and $d(x_{n+1},x^*)\to 0$ we have $$d(Tx^*,Tx_n)+d(Tx_{n+1},Tx^*) \to 0$$ and, hence, $d(Tx^*,Tx_n) \to 0$.

Let now $(x_n)$ be a sequence such that $x_n\to x^*$, and
$x_{n}\neq x^*$ for all $n$, but $x_{n} = x_{n+1}$ is possible. Let $(x_{n_k})$ be a subsequence of $(x_n)$ obtained by deleting corresponding repeating elements of $(x_n)$, i.e.,  such that $x_{n_k}\neq x_{n_{k+1}}$ for all $k$. It is clear that $x_{n_k}\to x^*$. As was just proved $Tx_{n_k}\to Tx^*=x^*$. The difference between $Tx_{n_k}$ and $Tx_{n}$ is that $Tx_{n}$ can be obtained from $Tx_{n_k}$ by inserting corresponding repeating consecutive elements. Hence, it is easy to see that $Tx_{n}\to Tx^*$.

Let $(x_n)$ be a sequence such that $x_n=x^*$ for all $n>N$, where $N$ is some natural number. Then, clearly, $Tx_{n}\to Tx^*$.
Let $(x_n)$ now be an arbitrary sequence such that $x_n\to x^*$ but not like in the previous case. Consider a subsequence  $(x_{n_k})$ obtained from $(x_n)$ by deleting elements $x^*$ (if they exist). Clearly, $x_{n_k}\to x^*$. It was just shown that such $Tx_{n_k}\to Tx^*$. Again, we see that $Tx_n$ can be obtained from $Tx_{n_k}$ by inserting in some places elements $Tx^*=x^*$. Again, it is easy to see that $Tx_{n}\to Tx^*$.
\end{proof}

\section{Fixed point theorem in incomplete metric space}\label{sec5}

In the following theorem, we exclude the completeness requirement of the metric space, compare with Theorem 1 from~\cite{Ka69}.

\begin{thm}\label{t2}
Let $(X,d)$, $|X|\geqslant 3$, be a metric space and let the mapping $T\colon X\to X$ satisfy the following four conditions:
\begin{itemize}
  \item [(i)] $T$ does not possess periodic points of prime period $2$.
  \item [(ii)] $T$ is a generalized Chatterjea  type mapping on $X$.
  \item [(iii)] $T$ is continuous at $x^*\in X$.
  \item [(iv)] There exists a point $x_0\in X$ such that the sequence of iterates $x_n=Tx_{n-1}$, $n=1,2,...$, has a subsequence $x_{n_k}$, converging to $x^*$.
\end{itemize}
Then $x^*$ is a fixed point of $T$. The number of fixed points is at most two.
\end{thm}
\begin{proof}
Since $T$ is continuous at $x^*$  and $x_{n_k} \to x^*$ we have $Tx_{n_k}=x_{n_{k}+1} \to Tx^*$. Note that $x_{n_{k}+1}$ is a subsequence of $x_n$ but not obligatory the subsequence of $x_{n_k}$. Suppose $x^*\neq Tx^*$. Consider two balls $B_1=B_1(x^*,r)$ and $B_2=B_2(Tx^*,r)$, where $r<\frac{1}{3}d(x^*, Tx^*)$. Consequently, there exists a positive integer $N$ such that $i>N$ implies
$$
x_{n_i} \in B_1 \, \text{ and } \, x_{n_i+1} \in B_2.
$$
Hence,
\begin{equation}\label{w4}
d(x_{n_i}, x_{n_i+1})>r
\end{equation}
for $i>N$.

If the sequence $x_n$ does not contain a fixed point of the mapping $T$, then we can apply considerations of Theorem~\ref{t1}.
By~(\ref{e9}) for $n=3,4,\ldots$ we have
$$
d(x_{n-1},x_n)\leqslant \alpha^{\frac{n}{2}-1}a,
$$
where $a=\max\{d(x_{0},x_{1}),d(x_{1},x_{2})\}$ and $\alpha=2\gamma/(1-\gamma)\in [0,1)$.
Hence,
$$
d(x_{n_i},x_{n_i+1})\leqslant \alpha^{\frac{n_i+1}{2}-1}a.
$$
But the last expression approaches $0$ as $i\to \infty$ which contradicts to~(\ref{w4}). Hence, $Tx^*=x^*$.

The existence of at most two fixed points follows form the last paragraph of Theorem~\ref{t1}.
\end{proof}

In the following theorem we suppose that $T$ is a generalized Chatterjea  type mapping not on $X$ but on everywhere dense subset of $X$ and suppose that $T$ is continuous on $X$ but not only at the point $x^*$, compare with Theorem 2 from~\cite{Ka69}.

\begin{thm}\label{t3}
Let $(X,d)$, $|X|\geqslant 3$, be a metric space and let the mapping $T\colon X\to X$ be continuous. Suppose that
\begin{itemize}
  \item [(i)] $T$ does not possess periodic points of prime period $2$.
  \item [(ii)] $T$ is a generalized Chatterjea type mapping on $(M,d)$, where $M$ is an everywhere dense subset of $X$.
  \item [(iii)] There exists a point $x_0\in X$ such that the sequence of iterates $x_n=Tx_{n-1}$, $n=1,2,...$, has a subsequence $x_{n_k}$, converging to $x^*$.
\end{itemize}
Then $x^*$ is a fixed point of $T$. The number of fixed points is at most two.
\end{thm}
\begin{proof}
The proof will follow from Theorem~\ref{t2}, if we can show that $T$ is a generalized Chatterjea type mapping on $X$. Let $x, y, z$ be any three pairwise distinct points of $X$ such that  $x, y \in M$, $z\in X \setminus M$ and  let $(c_n)$ be a sequence in $M$ such that $c_n\to z$, $c_n\neq x$, $c_n\neq y$ for all $n$ and $c_i\neq c_j$, $i\neq j$. Then
$$
 d(Tx,Ty)+d(Ty,Tz)+d(Tx,Tz) \leqslant
$$
$$
d(Tx,Ty)+d(Ty,Tc_n)+d(Tc_n,Tz)+d(Tx,Tc_n)+d(Tc_n,Tz)
$$
$$
\leqslant \gamma(d(x,Ty)+d(x,Tc_n)+d(y,Tx)+ d(y,Tc_n)+d(c_n,Tx)+d(c_n,Ty))+2d(Tc_n,Tz)
$$
(in view of the inequalities
\begin{equation} \label{s1}
\begin{aligned}
d(x,Tc_n) &\leqslant d(x,Tz)+d(Tz,Tc_n),\\
d(y,Tc_n) &\leqslant d(y,Tz)+d(Tz,Tc_n),\\
d(c_n,Tx) &\leqslant d(c_n,z)+d(z,Tx),\\
d(c_n,Ty) &\leqslant d(c_n,z)+d(z,Ty).
\end{aligned}
\end{equation}
we get)
\begin{align*}
&\leqslant \gamma(d(x,Ty)+d(x,Tz)+d(y,Tx)+d(y,Tz) \\
&\quad + d(z,Tx)+d(z,Ty))+2\gamma d(c_n,z)+ (2+\gamma)d(Tz,Tc_n).
\end{align*}

Letting $n\to \infty$, we get $d(c_n,z)\to 0$ and $d(Tc_n,Tz)\to 0$. Hence, inequality~(\ref{gctm}) follows.

Let now  $x \in M$, $y, z\in X \setminus M$, and let $(b_n), (c_n)$ be sequences in $M$ such that $b_n\to y$ and $c_n\to z$. (Here and below we consider that the points $x, y, z$ and all elements of sequences converging to these points are pairwise distinct.) Then
\begin{align*}
&d(Tx,Ty)+d(Ty,Tz)+d(Tx,Tz) \\
&\leqslant d(Tx,Tb_n)+d(Tb_n,Ty)+d(Ty,Tb_n)+d(Tb_n,Tc_n)+d(Tc_n,Tz) \\
&+d(Tx,Tc_n)+d(Tc_n,Tz) \\
&\leqslant \gamma(d(x,Tb_n)+d(x,Tc_n)+d(b_n,Tx)+d(b_n,Tc_n)
 +d(c_n,Tx)+d(c_n, Tb_n))\\ &\quad+ 2d(Tb_n,Ty) + 2d(Tc_n,Tz).
\end{align*}

(in view of the inequalities
\begin{equation} \label{s2}
\begin{aligned}
d(x,Tb_n) &\leqslant d(x,Ty)+d(Ty,Tb_n),\\
d(x,Tc_n) &\leqslant d(x,Tz)+d(Tz,Tc_n),\\
d(b_n,Tx) &\leqslant d(b_n,y)+d(y,Tx),\\
d(b_n,Tc_n) &\leqslant d(b_n,y)+d(y,Tz)+d(Tz,Tc_n),\\
d(c_n,Tx) &\leqslant d(c_n,z)+d(z,Tx), \\
d(c_n,Tb_n) &\leqslant d(c_n,z)+d(z,Ty)+d(Ty,Tb_n).
\end{aligned}
\end{equation}
we get)

\begin{align*}
&\leqslant \gamma(d(x,Ty)+d(x,Tz)+d(y,Tx)+d(y,Tz) \\
&\quad + d(z,Tx)+d(z,Ty)) + 2(1+\gamma)d(Tb_n,Ty) + 2(1+\gamma)d(Tc_n,Tz) \\
&+ 2\gamma(d(b_n,y)+d(c_n,z)).
\end{align*}

Again, letting $n\to \infty $, we get inequality~(\ref{gctm}).

Let now $x, y, z\in X \setminus M$, and let $(a_n), (b_n)$ and $(c_n)$ be sequences in $M$ such that $a_n\to x$, $b_n\to y$ and $c_n\to z$.
Then
$$
d(Tx,Ty)+d(Ty,Tz)+d(Tx,Tz)
$$
$$
\leqslant d(Tx,Ta_n)+d(Ta_n,Tb_n)+d(Tb_n,Ty)
$$
$$
+d(Ty,Tb_n)+d(Tb_n,Tc_n)+d(Tc_n,Tz)
$$
$$
+d(Tx,Ta_n)+d(Ta_n,Tc_n)+d(Tc_n,Tz)
$$
$$
\leqslant \gamma(d(a_n,Tb_n)+d(a_n,Tc_n)+d(b_n,Ta_n)+d(b_n,Tc_n)+d(c_n,Ta_n)+d(c_n,Tb_n))
$$
$$
+2d(Ta_n,Tx)
+2d(Tb_n,Ty)
+2d(Tc_n,Tz)
$$
(using the inequalities
\begin{align*}
d(a_n,Tb_n) &\leqslant d(a_n,x) + d(x,Ty) + d(Ty,Tb_n) \\
d(a_n,Tc_n) &\leqslant d(a_n,x) + d(x,Tz) + d(Tz,Tc_n) \\
d(b_n,Ta_n) &\leqslant d(b_n,y) + d(y,Tx) + d(Tx,Ta_n) \\
d(b_n,Tc_n) &\leqslant d(b_n,y) + d(y,Tz) + d(Tz,Tc_n) \\
d(c_n,Ta_n) &\leqslant d(c_n,z) + d(z,Tx) + d(Tx,Ta_n) \\
d(c_n,Tb_n) &\leqslant d(c_n,z) + d(z,Ty) + d(Ty,Tb_n)
\end{align*}

we get)
$$
\leqslant \gamma(d(x,Ty)+d(x,Tz)+d(y,Tx)+d(y,Tz)+d(z,Tx)+d(z,Ty))
$$
$$
+2(1+\gamma)d(Ta_n,Tx)
+2(1+\gamma)d(Tb_n,Ty)
+2(1+\gamma)d(Tc_n,Tz)
$$
$$
+2\gamma(
d(a_n,x)+ d(b_n,y)+d(c_n,z)).
$$
Again, letting $n\to \infty$, we get inequality~(\ref{gctm}). Hence, $T$ is a generalized Chatterjea  type mapping on $X$, which completes the proof.
\end{proof}











\end{document}